\documentclass[12pt, reqno]{amsart}

\usepackage{amsmath, amsthm, amssymb, stmaryrd}
\usepackage{hyperref}
\usepackage{enumerate}
\usepackage{url}

\usepackage{bm}

\usepackage{xcolor}

\usepackage{fancyvrb}

\newtheorem{thm}{Theorem}[section]
\newtheorem{lemma}[thm]{Lemma}
\newtheorem{prop}[thm]{Proposition}
\newtheorem{cor}[thm]{Corollary}

\theoremstyle{definition}
\newtheorem{defn}[thm]{Definition}

\theoremstyle{remark}
\newtheorem{remark}[thm]{Remark}

\numberwithin{equation}{section}

\theoremstyle{definition}
\newtheorem{prob}{Problem}

\DeclareMathOperator{\Mod}{mod}
\newcommand{\mmod}[1]{\;(\Mod{ #1})}

\def\alp{{\alpha}} 
  
\def\gam{{\gamma}} 
\def\del{{\delta}} 
 
\def\tet{{\theta}}  

\def\lam{{\lambda}}

\def\eps{\varepsilon}
\def\implies{\Rightarrow}

\def\le{\leqslant} \def\ge{\geqslant}

\def\d{{\,{\rm d}}}

\def \bC {\mathbb C}

\def \bN {\mathbb N}

\def \bR {\mathbb R}
\def \bZ {\mathbb Z}
\def \bT {\mathbb T}

\def \bc {\mathbf c}

\def \bn {{\mathbf{n}}}

\def \bx {\mathbf x}

\def \bmu {{\boldsymbol{\mu}}}

\def \cA {\mathcal A}

\def \cK {\mathcal K}

\def \cP {\mathcal P}

\def \cR {\mathcal R}

\def \meas {\mathrm{meas}}

\def \PS {\mathrm{PS}}
\def \supp {{\mathrm{supp}}}

\def \N {\mathbb{N}}

\def \Z {\mathbb{Z}}

\begin{document}
\title[Schur--Piatetski-Shapiro]{Additive Ramsey theory over Piatetski-Shapiro numbers}
\author{Jonathan Chapman \and Sam Chow \and Philippa Holdridge}

\address{Mathematics Institute, Zeeman Building, University of Warwick, Coventry CV4 7AL, United Kingdom}
\email{Jonathan.Chapman@warwick.ac.uk}


\address{Mathematics Institute, Zeeman Building, University of Warwick, Coventry CV4 7AL, United Kingdom}
\email{Sam.Chow@warwick.ac.uk}

\address{Mathematics Institute, Zeeman Building, University of Warwick, Coventry CV4 7AL, United Kingdom}
\email{holdridge.philippa@renyi.hu}

\begin{abstract} We characterise partition regularity for linear equations over the Piatetski-Shapiro numbers $\lfloor n^c \rfloor$ when $1 < c < c^\dag(s)$, where $s \ge 3$ is the number of variables. Here $c^\dag(3) = 12/11$ and $c^\dag(4) = 7/6$, while $c^\dag(s) = 2$ for $s \ge 5$. We also establish density results with quantitative bounds. Following recent developments, we take this opportunity to update Browning and Prendiville's version of Green's Fourier-analytic transference principle, strengthening its conclusion.
\end{abstract}

\maketitle

\section{Introduction}

\subsection{Monochromatic Schur triples}

A \emph{Schur triple} is $(x,y,z) \in \bN^3$ such that $x + y = z$. Schur \cite{Sch1916} showed that if $\bN$ is partitioned into finitely many colour classes then there exists a monochromatic Schur triple. We state Schur's result below, writing $[r]$ for the set of positive integers $n \le r$.

\begin{thm} [Schur, 1916]
\label{thm1.1}
Suppose $\bN = C_1 \cup \cdots \cup C_r$. Then there exist $k \in [r]$ and $x, y, z \in C_k$ such that $x + y = z$.
\end{thm}

This was generalised by Rado \cite{Rad1933} to systems of linear equations. A system of equations is \emph{partition regular} if, for any partition of $\bN$ into finitely many colour classes, it contains a non-trivial solution for which all of the variables lie in the same colour class. A solution is \emph{non-trivial} if the variables are pairwise distinct. We now state Rado's criterion in the case of a single equation.

\begin{thm}
[Rado, 1933]
Let $s \ge 3$ and $c_1, \ldots, c_s$ be non-zero integers. Then the equation $\bc \cdot \bx = 0$ is partition regular if and only if $\sum_{i \in I} c_i = 0$ for some non-empty $I$.
\end{thm}

Erd\H{o}s and Graham \cite{Gra2007, Gra2008} asked whether there are monochromatic Schur triples in the squares --- i.e. monochromatic Pythagorean triples --- offering $\$250$ for an answer.

\begin{prob}
\label{probA}
Prove or disprove that if $\bN = C_1 \cup \cdots \cup C_r$ then there exist $k \in [r]$ and $x, y, z \in C_k$ such that $x^2 + y^2 = z^2$.
\end{prob}

A computer-assisted proof was found in the case of two colours \cite{HKM2016} which, at the time, was the largest mathematical proof ever \cite{Lam2016}. Partition regularity for generalised Pythagorean triples in five or more variables, i.e. solutions to
\[
x_1^2 + \cdots + x_{s-1}^2 = x_s^2
\]
where $s \ge 5$, was demonstrated by Chow, Lindqvist and Prendiville~\cite{CLP2021}. A higher-degree analogue of Rado's criterion was also provided therein, and the result was generalised in \cite{Cha2022, CC1, CC2}. Equations that are quadratic in some variables and linear in others have also received significant attention \cite{GL2019, Mor2017, Pac2018, Pre2021}.
Recently, it was shown in \cite{FKM} that Pythagorean triples exist in which two of the variables have the same colour.

\begin{thm} [Frantzikinakis--Klurman--Moreira, 2023+]
If 
\[
\bN = C_1 \cup \cdots \cup C_r
\]
then:
\begin{enumerate} [(i)]
\item There exists $u \in [r]$ such that $x^2 + y^2 = z^2$ for some $x,y \in C_u$ and some $z \in \bN$.
\item There exists $v \in [r]$ such that $x^2 + y^2 = z^2$ for some $x,z \in C_v$ and some $y \in \bN$.
\end{enumerate}
\end{thm}

The \emph{Piatetski-Shapiro numbers} with parameter $c > 1$ form a set
\[
\PS_c = \{ \lfloor n^c \rfloor: n \in \bN \}.
\]
Note that $\PS_2$ is the set of positive squares. When $1 < c < 2$, the Piatetski-Shapiro numbers behave similarly to the squares in terms of growth and separation, but the squares are arithmetically richer. We solve a version of Problem \ref{probA} for Piatetski-Shapiro numbers. For $s \ge 3$ an integer, let
\begin{equation}
\label{eqn1.1}
c^*(s) = 
\frac{2s + 6}{s + 8}, \qquad
c^\dag(s) = \begin{cases}
c^*(s), &\text{if } s = 3, 4 \\
2, &\text{if } s \ge 5.
\end{cases}
\end{equation}
For $c > 0$ and $N \ge 1$, let
\[
\PS_c(N) = \{ n \in \PS_c: n \le N \}.
\]

\begin{thm} 
\label{thm1.4}
Let $s \ge 3$ be an integer, and let
$1 < c < c^\dag(s)$. Let $c_1, \ldots, c_s$ be non-zero integers with $\sum_{i \in I} c_i = 0$, for some non-empty
$I$. Let $N \in \bN$ be large, and suppose
\[
\PS_c(N) = C_1 \cup \cdots \cup C_r.
\]
Then there exist $k \in [r]$ and $\bx \in C_k^s$ such that
\begin{equation}
\label{eqn1.2}
\bc \cdot \bx = 0, \qquad
x_i \ne x_j \quad (i \ne j).
\end{equation}
\end{thm}

\begin{cor} \label{cor1.5}
If $1 < c < 12/11$, and $\bN$ is partitioned into finitely many colour classes, then $\PS_c$ contains a monochromatic Schur triple.
\end{cor}

\noindent Coincidentally, Piatetski-Shapiro \cite{PS1953} also obtained the threshold $12/11$ when proving a version of the prime number theorem for $\PS_c$. 

We thank Sean Prendiville for drawing our attention to the following more general problem, which he attributes to Frantzikinakis, and which inspired the present work. Note that the case $c = 1$ is solved by Theorem \ref{thm1.1}, and that the case $c = 2$ is Problem \ref{probA}.

\begin{prob} 
\label{probB}
Prove or disprove that the conclusion of Corollary \ref{cor1.5} holds for any $c \in [1, 2]$.
\end{prob}

\subsection{Relative sparsity of 3AP-free sets}

In a very influential paper, Roth \cite{Rot1953} showed that if $N$ is large and $A \subseteq [N]$ with $|A| \gg N$ then $A$ contains non-trivial three-term arithmetic progressions, i.e. non-trivial solutions to
$x + y = 2z$. More precisely, Roth proved that if $A\subseteq [N]$ does not contain any non-trivial three-term arithmetic progressions, then $|A| \ll N/\log\log N$.

Recently, there have been strong quantitative refinements of Roth's theorem by Kelley--Meka \cite{KM2023} and Bloom--Sisask \cite{BS}. To state them, we say that $A \subset \bN$ is \emph{3AP-free} if $x + y = 2z$ has no non-trivial solutions with $x,y,z \in A$. Remarkably, the bound has the same shape as the best known lower bound for the largest 3AP-free set, which comes from a convexity construction going back to Behrend~\cite{Beh1946}.

\begin{thm}
[Kelley--Meka, 2023 and Bloom--Sisask, 2023+] If \mbox{$N \ge 2$} and $A \subset [N]$ is 3AP-free, then
\[
|A| \le \exp(-c \log^{1/9} N) N,
\]
for some constant $c > 0$.
\end{thm}

Turning to the non-linear setting, consider the following density analogue of Problem \ref{probA}.

\begin{prob}
\label{probC}
Prove or disprove that if $A$ is a 3AP-free subset of the squares then $|A \cap [N]| = o(\sqrt N)$.
\end{prob}

\noindent
Equivalently, show that
$
x^2 + y^2 = 2z^2
$
is density regular, meaning that any $A\subseteq[N]$ which lacks non-trivial solutions to this equation must satisfy $|A|=o(N)$.

For diagonal quadratic equations in five or more variables, such as
\[
x_1^2 + \cdots + x_4^2 = 5x^2,
\]
quantitative density regularity was characterised by Browning and Prendiville \cite{BP2017}. 

\begin{thm}[Browning--Prendiville, 2017]\label{thm1.7}
Let $s \ge 5$ and $c_1, \ldots, c_s$ be non-zero integers with $c_1 + \cdots + c_s = 0$.
Let $N \in \bN$ be large,
and suppose $A \subseteq \PS_2(N)$ is such that there does not exist $\bx\in A^s$ satisfying \eqref{eqn1.2}. Then, for any $\eps>0$,
\begin{equation*}
|A|\ll_{\bc,\eps} \frac{|\PS_2(N)|}{(\log\log\log N)^{(s - 2 -\eps)/2}}.
\end{equation*}
\end{thm}

Browning and Prendiville's techniques were subsequently generalised to higher perfect powers and to primes in \cite{Cho2018}, and generalised further in \cite{Cha2022, CC1, CC2}.
By adapting these methods, we solve a version of Problem~\ref{probC} for Piatetski-Shapiro numbers, with a quantitative density bound.

\begin{thm} 
\label{thm1.8}
Let $\tau(3)=9$, and set $\tau(s)=7$ for all $s\geqslant 4$.
Let $s \ge 3$ be an integer, and let
$1 < c < c^\dag(s)$.
Let $c_1, \ldots, c_s$ be non-zero integers with
$c_1 + \cdots + c_s = 0$.
Let $N \in \bN$ be large. If $A \subseteq \PS_c(N)$ is such that there does not exist $\bx\in A^s$ satisfying \eqref{eqn1.2}, then
\begin{equation*}
|A| \leqslant \frac{|\PS_c(N)|}{\exp\left(C^{-1}(\log\log N)^{1/\tau(s)}\right)},
\end{equation*}
for some constant $C>1$ which depends only on $c$ and $\bc$.
\end{thm}

\begin{remark} Our density saving is substantially greater than the saving of $(\log \log \log N)^{O(1)}$ that Browning and Prendiville obtained in the case of squares. This is partly because the $W$-trick is not required in our setting. We make further gains by strengthening the Fourier-analytic transference principle \cite[Proposition 2.8]{BP2017} in \S \ref{sec4}. We do so by incorporating recent developments \cite{FHHK2024, Kos}.

In contrast with Theorem \ref{thm1.7}, we only need to assume that $s \ge 3$. With our constraint on $c$, we are able to obtain more precise estimates for exponential sums. This leads to a sharp restriction estimate at a lower exponent --- see \S \ref{subsec1.3} --- which ultimately reduces the number of variables that we require.
\end{remark}

Note that the natural analogue of Problem \ref{probC} for Piatetski-Shapiro sequences with parameter $c \in [1,2)$ has already been settled. A result of Saito and Yoshida \cite[Corollary 5]{SY2019} --- which uses Szemer\'edi's theorem \cite{Sze1975} to find progressions in graphs of ``slightly curved sequences'' --- implies the following.

\begin{thm} [Saito--Yoshida, 2019]
Fix $c \in [1,2)$ and an integer $k \ge 3$. Let $N \in \bN$ be large, and suppose $A \subseteq \PS_c(N)$ with 
\[
|A| \gg |\PS_c(N)|.
\]
Then $A$ contains a non-trivial $k$-term arithmetic progression.
\end{thm}

One can deduce the following application using the further observation that if $c_1 + \cdots + c_s = 0$ then, for any $\lam \in \bR$ and $\mu_1 = \cdots = \mu_s \in \bR$,
\[
\bc \cdot \bx = 0 \implies \bc \cdot (\lam \bx + \bmu) = 0 \qquad (\bx \in \bR^s).
\]

\begin{cor} 
Let $s \ge 3$ be an integer, and let
$1 < c < 2$.
Let $c_1, \ldots, c_s$ be non-zero integers with
$c_1 + \cdots + c_s = 0$.
Let $N \in \bN$, and suppose $A \subseteq \PS_c(N)$ is such that there does not exist $\bx\in A^s$ satisfying \eqref{eqn1.2}. Then
$
|A| = o(|\PS_c(N)|)
$ 
as $N \to \infty$.
\end{cor}

Compared with these wonderful findings, Theorem \ref{thm1.8} has the advantage of providing a quantitative density bound. 

\bigskip

One can also consider these problems for Piatetski-Shapiro primes, i.e. primes that are Piatetski-Shapiro numbers. Mirek \cite{Mir2015} achieved this in the case of 3AP-free sets, but required that $c < 72/71$. We state Mirek's result below, writing $\cP$ for the set of primes.

\begin{thm}
[Mirek, 2015]
Let $1 < c < 72/71$. If $N$ is large and $A$ is a 3AP-free subset of $\PS_c(N) \cap \cP$, then $|A| = o(|\PS_c(N) \cap \cP|)$.
\end{thm}

One should think twice before looking for monochromatic Schur triples of Piatetski-Shapiro primes, however, as there are no Schur triples of odd numbers. Using our methods and the ingredients \cite[Theorem 1.8 and Lemma 1.10]{Mir2015}, one can show that if $1 < c < 72/71$ and $\PS_c \cap \cP = C_1 \cup \cdots \cup C_r$ then there exist $k \in [r]$ and $x,y,z \in C_k$ such that
\[
x + y = z + 1.
\]
See \cite{CC1, Le2012, ZZ2025} for some similar results and open problems. We leave it to the interested reader to explore further generalisations.

\subsection{Methods}
\label{subsec1.3}

We use the Fourier-analytic transference principle, as pioneered by Green \cite{Gre2005}. 
We therefore introduce a weight function $\nu$ supported on $\PS_c(N)$ such that $\| \nu \|_1 \sim N$ as $N \to \infty$. The idea is to compare the $\nu$-weighted count of solutions with the $1_{[N]}$-weighted count, essentially by combining some estimates in Fourier space. The latter count was estimated by Frankl, Graham and R\"odl \cite{FGR1988}. The key ingredients for the density result, Theorem \ref{thm1.8}, are as follows.
\begin{enumerate}
\item Fourier decay: For some $\chi \in (0,1)$, we have
\[
\| \hat \nu - \hat 1_{[N]} \|_\infty \ll_\chi N^{1-\chi}.
\]

\item Restriction estimate: For some $t \in (s-1,s)$, we have
\[
\sup_{|\psi| \le \nu} \int_\bT |\hat \psi(\alp)|^t \d \alp \ll_t N^{t-1}.
\]

\item Trivial count:
\[
\sum_{\bn \in \cK} \prod_{i \le s} \nu(n_i) = o(N^{s-1}) \qquad (N \to \infty),
\]
where 
\[
\cK = \{ \bn \in \bZ^s: \bc \cdot \bn = 0, \quad
n_i = n_j \text{ for some } i \ne j \}.
\]
\end{enumerate}
There is one other ingredient, \emph{density transfer}, which roughly asserts that
\[
\sum_{n \in A} 1_{[N]}(n) \gg N \quad \Longrightarrow \quad \sum_{n \in A} \nu(n) \gg N,
\]
but this is easy to confirm. The Ramsey-theoretic result, Theorem \ref{thm1.4}, requires further combinatorial machinery, which we import from \cite{CLP2021}.

We gauge the strength of our results by the value of $c^\dag(s)$. This depends on the values of $\chi$ and $t$ relating to 
Fourier decay and restriction, so we have attempted to maximise $\chi$ and minimise $t$. 
Although we are able to obtain power-saving Fourier decay for all $1<c<2$ in Proposition \ref{prop2.1}, we do need to assume that $c < c^\dag(s)$ in order to also obtain the restriction estimate \eqref{eqn3.3}, which is a limitation when $s \le 4$.

\subsection*{Organisation}

We establish a Fourier decay estimate in Section \ref{sec2}, and a restriction estimate in Section \ref{sec3}. In \S \ref{sec4}, we strengthen
the transference principle provided by Browning in Prendiville in \cite[Proposition 2.8]{BP2017}, using recent developments. Finally, in Section \ref{sec5}, we prove Theorems~\ref{thm1.4} and \ref{thm1.8}.

\subsection*{Notation}

For $x \in \bR$, we write $e(x) = e^{2 \pi i x}$. We put $\bT = [0,1]$. For $f \in \ell^1(\bZ)$, we define the Fourier transform by
\[
\hat f(\alp) = \sum_{n \in \bZ} f(n) e(n \alp).
\]

We employ the Vinogradov and Bachmann--Landau asymptotic notations, as we now describe. For complex-valued functions $f$ and $g$, we write $f \ll g$ or $f = O(g)$ if there exists a constant $C$ such that $|f| \le C|g|$ pointwise. We indicate the dependence of the implicit constant $C$ on some parameters $\lam_1, \ldots, \lam_t$ using subscripts, for example $f \ll_{\lam_1, \ldots, \lam_t} g$ or $f = O_{\lam_1, \ldots, \lam_t}(g)$. We write $f \asymp g$ if $f \ll g \ll f$. We write $f = o(g)$ if $f/g \to 0$ and $f \sim g$ if $f/g \to 1$, in some specified limit.

\subsection*{Funding}

JC was supported by the Heilbronn Institute for Mathematical Research, and is currently supported by EPSRC through Joel Moreira's Frontier Research Guarantee grant, ref. \texttt{EP/Y014030/1}. PH is supported by the Warwick Mathematics Institute Centre
for Doctoral Training, and gratefully acknowledges funding from the University of Warwick.

\subsection*{Rights}

For the purpose of open access, the authors have applied a
Creative Commons Attribution (CC-BY) licence to any Author Accepted Manuscript version arising from this submission.

\section{Fourier decay}
\label{sec2}

The purpose of this section is to prove the following Fourier decay estimate for our weight function 
\[
\nu(n) = \nu_c(n):= \phi'(n)^{-1} 1_{\PS_c(N)}(n),
\]
where $\phi(n) = n^{1/c}$.

\begin{prop} [Fourier decay]
\label{prop2.1}
Let $c \in (1,2)$ and $\eps > 0$. Then
\[
\| \hat \nu - \hat 1_{[N]} \|_\infty
\ll_{c, \eps} N^{\eps + 6/5 - 2/(5c)}.
\]
\end{prop}

The broad strategy used to understand the Fourier coefficients of $\nu$ is given in \cite{HB1983}, where it is attributed to Vaughan.

\begin{lemma}
For all $c>1$ and $n\in\N$,
\[
1_{\PS_c}(n) = \lfloor -\phi(n) \rfloor - \lfloor - \phi(n+1) \rfloor.
\]
\end{lemma}

\begin{proof}
Since $c>1$, the mean value theorem implies that the function
\begin{equation*}
    n \mapsto \lfloor -\phi(n) \rfloor - \lfloor - \phi(n+1) \rfloor
\end{equation*}
has image $\{ 0,1 \}$. Observe that $n\in\N$ is mapped to $1$ by this function if and only if there exists $m\in\N$ such that
\begin{equation*}
    \phi(n)\leqslant m < \phi(n+1).
\end{equation*}
Applying $\phi^{-1}$, this is equivalent to
\begin{equation*}
    n \leqslant m^c < n+1,
\end{equation*}
which holds if and only if $n=\lfloor m^c\rfloor$.
\end{proof}

An application of this lemma reveals that
\begin{equation}
\label{eqn2.1}
\hat{\nu}(\alpha) = \sum_{n \le N}
\phi'(n)^{-1} e(n \alp)
(\lfloor -\phi(n) \rfloor - \lfloor - \phi(n+1) \rfloor).
\end{equation}
To further modify this expression, we deploy the saw-tooth function
\begin{equation*}
    \psi(x) := x-\lfloor x\rfloor - 1/2.
\end{equation*}
For each $n\in\N$, observe that
\[
\lfloor -\phi(n) \rfloor - \lfloor - \phi(n+1) \rfloor =
\phi(n+1) - \phi(n) + \psi(-\phi(n+1)) - \psi(-\phi(n)).
\]
By Taylor's theorem with Lagrange remainder, there exists $\xi_n\in[0,1]$ such that
\[
\phi(n+1) - \phi(n) = 
\phi'(n) + \phi''(n+\xi_n)/2.
\]
Thus,
\begin{align*}
&\sum_{n \le N}
\phi'(n)^{-1} e(n \alp)
(\phi(n+1) - \phi(n)) \\ &= 
\sum_{n \le N}
e(n \alp)\left( 1 + \frac{(1-c)(n+\xi_n)^{\frac{1}{c} - 2}}{2cn^{\frac{1}{c}-1}}\right)\\
& = \hat{1}_{[N]}(\alpha) + O(\log N).
\end{align*}
Inserting this into (\ref{eqn2.1}), we find that
\begin{equation*}
    \hat{\nu}(\alpha) - \hat{1}_{[N]}(\alpha) = \sum_{n \le N} 
\frac{\psi(-\phi(n+1)) - \psi(-\phi(n))}{\phi'(n)}
e(n\alp) + O(\log N).
\end{equation*}

To make further progress, we split the sum on the right-hand side into dyadic ranges:
\begin{equation}
\label{eqn2.2}
\hat{\nu}(\alpha) - \hat{1}_{[N]}(\alpha) = \sum_{k \le \frac{\log N}{\log 2}} S(N/2^k) + O(\log N),
\end{equation}
where
\[
S(P) = S(P;\alpha) = \sum_{P < n \le 2P} 
\frac{\psi(-\phi(n+1)) - \psi(-\phi(n))}{\phi'(n)}
e(n\alp).
\]
To prove Proposition \ref{prop2.1}, we require estimates for the exponential sums $S(P)$. As noted in \cite[\S2]{HB1983}, we have the following standard approximation for $\psi$ by trigonometric sums.

\begin{lemma}
\label{lem2.3}
For any integer $M \ge 2$,
\begin{align*}
\psi(t) - \sum_{0 < |m| \le M} \frac{e(-mt)}{2 \pi i m}  \ll
\min \{1, (M\| t \|)^{-1} \} = \sum_{m\in\bZ} b_{m}e(mt),
\end{align*}
for some 
\[
b_{m} \ll \min \left \{ \frac{\log M}{M}, \frac{1}{|m|}, \frac{M}{m^{2}} \right \}.
\]
\end{lemma}

We apply this lemma with $M\in[2,P]$ equal to a power of $P$ to be specified later. This gives
\[
S(P) - S_0 \ll S_1 + S_2,
\]
where
\begin{align*}
S_0 &= \sum_{0 < |m| \le M} \frac1{2 \pi i m} \sum_{P < n \le 2P} \frac{e(n \alp + m \phi(n+1)) - e(n \alp + m \phi(n))}{\phi'(n)}, \\
S_1 &= \sum_{P < n \le 2P } 
\frac{\min\{1, (M \|\phi(n)\|)^{-1} \}}{\phi'(n)}, \\
S_2 &= \sum_{P < n \le 2P}  
\frac{\min\{1, (M \|\phi(n+1)\|)^{-1} \}}{\phi'(n)}.
\end{align*}

We begin by bounding $S_0$. Using partial summation, we find that
\begin{align*}
&\sum_{P < n \le 2P} \frac{e(n \alp + m \phi(n+1)) - e(n \alp + m \phi(n))}{\phi'(n)}\\&= U_{m}(2P)\varphi_{m}(2P) - \int_{P}^{2P} U_{m}(x) \varphi'_{m}(x) \d x, \end{align*}
where
\[
\varphi_{m}(x) = \frac{e(m(\phi(x+1) -\phi(x)))- 1}{\phi'(x)}
\]
and
\[
U_{m}(x) = \sum_{P < n \le x} e(n \alp + m\phi(n)).
\]
We compute that 
$
\varphi_{m}(x) \ll m
$ 
and
$
\varphi'_{m}(x) \ll m/x,
$
whence
\[
S_0 \ll 
\sum_{0 < |m| < M} \sup_{x \in (P,2P]} \left| U_{m}(x)\right|.
\]
Notice that, for $P < x \le 2P$, the function $F(x)=\alp x+ m\phi(x)$ satisfies $|F''(x)| \asymp |m|P^{1/c-2}$. Invoking van der Corput's lemma \cite[Lemma~3.1]{Mir2015}, we deduce that
\[U_{m}(x) \ll |m|^{1/2} P^{1/(2c)} + P^{1-1/(2c)} |m|^{-1/2},\]
and so
\[
S_{0} \ll M^{3/2} P^{1/(2c)} + M^{1/2} P^{1-1/(2c)}.
\]

We now consider the sum $S_1$. The key ingredient is the following estimate observed by Heath-Brown \cite[Lemma 1]{HB1983}.

\begin{lemma}
\label{lem2.4}
Let $m \in \bZ$ and $1/2 < \gam < 1$. Then, whenever we have $1 \le P \le P_1 \le 2P$, 
\begin{equation*}
\sum_{P < n \le P_1} e(m n^\gamma) \ll \min\{P, |m|^{-1} P^{1-\gamma} + (|m| P^\gamma)^{1/2}
\}.
\end{equation*}
\end{lemma}

To simplify our notation, let $\gamma = 1/c$, so that $\phi(n) = n^\gamma$. By the triangle inequality,
\begin{equation*}
\phi'(P) S_1 \le \sum_{m\in \bZ} |b_{m}|\left| \sum_{P < n \le 2P} e(m\phi(n)) \right|.
\end{equation*}
For further convenience, we write
\begin{equation*}
T_m = T_m(P) = \sum_{P < n \le 2P} e(m\phi(n)).
\end{equation*}
Our goal is to establish bounds for
\begin{equation*}
\sum_{m\in\bZ}|b_mT_m|.
\end{equation*}
We accomplish this by examining the contributions from various ranges for $m\in\bZ$. 

We begin with ``small'' $|m|$ for which the bound $b_m \ll (\log M)/M$ from Lemma \ref{lem2.3} suffices. The contribution from $m=0$ is immediate:
\begin{equation*}
|b_0 T_0| = |b_0|P\ll \frac{P\log M}{M}.
\end{equation*}
Notice that, for any $m,P \ge 1$, we have
\begin{equation*}
(m P^\gamma)^{1/2} \le m^{-1} P^{1-\gamma} 
\end{equation*}
if and only if
\begin{equation*}
m \leqslant P^{2/3 - \gamma}.
\end{equation*}
Let $L = \lfloor (2P)^{2/3 - \gamma}\rfloor$. By Lemma \ref{lem2.4},
\begin{align*}
\sum_{0<|m|\le L} |b_mT_m| &\ll \frac{\log M}{M} P^{1-\gamma} \sum_{0<|m|\le L} \frac{1}{|m|} \\ &\ll \frac{(\log M)\log L}{M} P^{1-\gamma} \ll_\eps \frac{P^{1-\gamma + \eps}}{M}.
\end{align*}
A similar computation reveals that
\begin{equation*}
\sum_{L<|m|\le M}|b_m T_m| \ll \frac{\log M}{M} P^{\gamma/2} \sum_{L<|m|\le M} |m|^{1/2} \ll M^{1/2} P^{\gamma/2} \log M.
\end{equation*}

For $|m| > M$, we use the bounds $|b_m| \ll M/m^2$ and $|T_m|^2\ll |m|P^{\gamma}$ from Lemmas \ref{lem2.3} and \ref{lem2.4}, respectively, to deduce that
\begin{equation*}
\sum_{|m| > M} |b_m T_m| \ll
M P^{\gam/2} \sum_{|m| > M} |m|^{-3/2} \ll M^{1/2} P^{\gamma/2}.
\end{equation*}
Since $\phi'(P) = \gamma P^{\gamma - 1}$, we can combine all of the above observations to see that
\[
\sum_{m \in \bZ} |b_m T_m| \ll (M^{-1}P + M^{1/2}P^{\gam/2}) \log M,
\]
whence
\begin{equation*}
S_1 \ll
(M^{-1}P^{2-\gamma} + M^{1/2}P^{1-\gamma/2}) \log M.
\end{equation*}
Using an almost identical argument, we can also show that
\begin{equation*}
S_2 \ll (M^{-1}P^{2-\gamma} + M^{1/2} P^{1-\gamma/2}) \log M.
\end{equation*}

We have therefore demonstrated that
\begin{equation*}
S(P) \ll M^{3/2} P^{\gam/2} + (M^{-1}P^{2-\gamma} + M^{1/2}P^{1-\gamma/2}) \log M.
\end{equation*}
Choosing $M = P^{4/5 - 3\gam/5}$ gives
\begin{equation*}
S(P) \ll_\eps P^{\eps + 6/5 - 2\gam/5}, 
\end{equation*}
for any $\eps > 0$. Inserting this estimate into (\ref{eqn2.2}) finishes the proof of Proposition \ref{prop2.1}.

\bigskip

Before moving on to the next section, we record the following observations. By Proposition \ref{prop2.1}, if $\eps > 0$ then
\[
\| \nu \|_1 = \hat \nu(0) = N + O(N^{\eps + 6/5 - 2/(5c)}).
\]
In particular, since $c < 2$,
\begin{equation}
\label{eqn2.3}
\| \hat \nu \|_\infty =
\| \nu \|_1 \sim N \qquad (N \to \infty).
\end{equation}

\section{Fourier restriction}
\label{sec3}

In this section, we establish two restriction estimates for our weight function $\nu=\nu_c$.

\begin{prop}
[Fourier restriction I]
\label{prop3.1}
Let $1 < c < 2$.
Suppose we have
\begin{equation}
\label{eqn3.1}
\| \hat \nu - \hat 1_{[N]} \|_\infty \ll_{c,\chi} N^{1-\chi},
\end{equation}
for some $\chi \in (0,1)$. Define $t_0 > 2$ by
\begin{equation}
\label{eqn3.2}
\chi = 2\frac{1-1/c}{t_0 - 2},
\end{equation}
and let $t > t_0$. Let $\psi: \bZ \to \bC$ with $|\psi| \le \nu$. Then
\begin{equation}
\label{eqn3.3}
\int_\bT |\hat \psi(\alp)|^t \d \alp
\ll_{c,\chi,t} N^{t-1}.
\end{equation}
\end{prop}

\begin{proof}
A na\"ive approach would be to consider $t=2$ and use orthogonality to obtain
\begin{equation}\label{eqn3.4}
\int_{\bT} |\hat{\psi}(\alp)|^{2}\d \alp \le \sum_{n \le N} \nu(n)^{2} \ll N^{2-1/c}.
\end{equation}
This only saves $1/c$ in the exponent. We need to save an additional exponent of $1-1/c$ and to do this, we use an epsilon removal process. This method is usually used to remove an arbitrarily small power $\eps$, but it can in fact achieve more, as in \cite[Appendix E]{CLP2021}.

Note from \eqref{eqn2.3} that $\| \hat \psi \|_\infty \ll N$.
For $0 < \del \ll 1$, define the large spectra by
\[
\mathcal{R}_{\delta} = \{ \alp \in \bT : |\hat{\psi}(\alp)| > \delta N \}.
\]
We claim that it suffices to prove that
\begin{equation}
\label{eqn3.5}
\mathrm{meas}
(\mathcal{R}_{\delta}) \ll_{\eps,t_0} \frac{1}{\delta^{t_{0}+\eps}N}
\end{equation}
for all sufficiently small $\eps>0$. To see this, let $C>0$ be such that $\sum_{n \le N} \nu(n) \le CN$ for all $N \ge 1$. Then
\[
\int_{\bT} |\hat{\psi}(\alp)|^{t} \d \alp \le \sum_{n=0}^{\infty} (2^{-n} CN)^{t} \mathrm{meas}(\mathcal{R}_{2^{-n}C}) \ll N^{t-1} \sum_{n=0}^{\infty} 2^{(t_{0}-t+\eps)n}.
\]
Then we simply take $\eps < t-t_{0}$.

Now, by \eqref{eqn3.4},
\[
\delta^{2}N^{2} \mathrm{meas}(\mathcal{R}_{\delta}) \le \int_{\bT} |\hat{\psi}(\alp)|^{2} \d \alp \ll N^{2-1/c},
\]
and therefore
\[
\mathrm{meas}(\mathcal{R}_{\delta}) \ll \delta^{-2}N^{-1/c}.
\]
So \eqref{eqn3.5} holds if
\[ \delta^{-2}N^{-1/c} \le \delta^{-t_{0}-\eps}N^{-1},
\]
which is equivalent to
\[
\delta \le N^{-\eta},
\]
where
\begin{equation}
\label{eqn3.6}
\eta = \frac{1-1/c}{t_{0}-2+\eps}.
\end{equation}
We may therefore suppose that
\begin{equation}
\label{eqn3.7}
N^{-\eta} < \delta \ll 1.
\end{equation}

Let $\theta_{1}, \ldots,\theta_{R} \in \cR_{\del}$ be such that $\|\theta_{i} -\theta_{j}\| \ge 1/N$ whenever $i \ne j$.
Let $c_{r}$ and $a_n$ be such that 
\[
|\hat{\psi}(\theta_{r})| = c_{r}\hat{\psi}(\theta_{r}), \qquad \psi(n) = a_{n} \nu(n),
\]
noting that $|c_r|, |a_n| \le 1$. Cauchy's inequality gives
\begin{align*}
\delta^{2}N^{2}R^{2} &\le \left( \sum_{r=1}^{R} |\hat{\psi}(\theta_{r})| \right)^{2} = \left(\sum_{r=1}^{R} c_{r}\sum_{n} a_{n} \nu(n)e(n\theta_{r})\right)^{2} \\ &\ll N \sum_{n} \nu(n) \left| \sum_{r=1}^{R} c_{r}e(n\theta_{r}) \right|^{2} \le N \sum_{1 \le r, r' \le R} |\hat{\nu}(\theta_{r}-\theta_{r'})|.
\end{align*}
Consequently
\[
\delta^{2}NR^{2} \ll \sum_{1 \le r, r' \le R} |\hat{\nu}(\theta_{r}-\theta_{r'})|
\]
and, by applying H\"older's inequality with $\gamma >2$ a constant to be chosen later,
\[
\delta^{2\gamma} N^{\gamma}R^{2} \ll \sum_{1 \le r, r' \le R} |\hat{\nu}(\theta_{r}-\theta_{r'})|^{\gamma}.
\]
It follows from \eqref{eqn3.1} and standard upper bounds for $\hat{1}_{[N]}$ that
\[
\hat{\nu}(\theta) \ll \frac{N}{1 + N \| \theta \|} + N^{1-\chi} \qquad (\tet \in \bT).
\]
    
Now
\[
\delta^{2\gamma} N^{\gamma}R^{2} \ll N^{\gamma} \sum_{1 \le r, r' \le R}  F(\theta_{r}-\theta_{r'}) + N^{\gamma(1-\chi)} R^{2},
\]
where $F(\tet) = (1 + N \| \tet \|)^{-\gam}$. Recalling \eqref{eqn3.2} and \eqref{eqn3.6}, we see that $\eta < \chi/2$. Combining this \eqref{eqn3.7}, we find that $N^{\gamma(1-\chi)}R^{2} = o(\delta^{2\gamma} N^{\gamma}R^{2})$. Therefore
\[ \delta^{2\gamma} R^{2} \ll \sum_{1 \le r, r' \le R}  F(\theta_{r}-\theta_{r'}).
\]
Then, by the argument in \cite[Section 4]{B1989}, but with $N^{2}$ replaced by $N$, the exponent $\gamma/2$ replaced by $\gam$ in the definition of $F$, and with $Q=1$, we obtain
\[
\delta^{2\gamma} R^{2} \ll 1.
\]
Hence $R \ll \delta^{-\gamma}$. Recalling that $\theta_{1},...,\theta_{R}$ are $1/N$-spaced points in $\mathcal{R}_{\delta}$, and supposing that this set is as large as possible, it follows that
\[
\mathrm{meas}(\mathcal{R}_{\delta}) \ll R/N \ll \frac{1}{\delta^{\gamma}N}.
\]
Choosing $\gamma = t_{0} + \eps$ then gives \eqref{eqn3.5} and completes the proof.
\end{proof}

Our second restriction estimate is similar, but we bootstrap a fourth moment estimate --- a bound on the additive energy --- instead of a second moment estimate. Specifically, we invoke \cite[Equation 47]{AEM2021}, which is essentially \cite[Theorem 2]{RS2006}. This tells us that
\[
\# \{ (x_1, \ldots, x_4) \in [X]^4:
|x_1^c + x_2^c - x_3^c - x_4^c| \le 2 \} \ll_{c,\eps} X^{2 + \eps} + X^{4-c+\eps},
\]
for any $c > 1$ and any $\eps > 0$.
Consequently, if $1 < c < 2$ and $\eps > 0$, then
\begin{align}
\notag E_c(N) &:=
\# \{ (n_1, \ldots, n_4) \in \PS_c(N)^4: n_1 + n_2 = n_3 + n_4 \}  \\ \label{eqn3.8}
&\ll_{c,\eps}
N^{\frac{4}{c} - 1 + \eps}.
\end{align}

\begin{prop}[Fourier restriction II]\label{prop3.2}
Let $1 < c < 2$.
Let $t > 4$, and suppose we have \eqref{eqn3.1} for some $\chi \in (0,1)$. Let $\psi: \bZ \to \bC$ with $|\psi| \le \nu$. Then we have \eqref{eqn3.3}.
\end{prop}

\begin{proof} We follow the proof of Proposition \ref{prop3.1}, \emph{mutatis mutandis}. We replace $t_0$ by 4. For any $\eps>0$, we see from \eqref{eqn3.8} that
\[
\del^4 N^4 \meas(\cR_\del) \le \int_\bT |\hat \psi(\alp)|^4 \d \alp \le \| \nu \|_\infty^4 E_c(N) \ll_\eps N^{3+\eps},
\]
so
\[
\meas(\cR_\del) \ll_\eps \del^{-4} N^{\eps-1}.
\]
Essentially the same proof then carries through.
\end{proof}

\section{An update on the transference principle}
\label{sec4}

In this section, we refine the general transference machinery given by Browning and Prendiville in \cite[Proposition 2.8]{BP2017}. Ko\'sciuszko mentioned this possibility in \cite{Kos}.
To elucidate the quantitative aspects of this result, we begin by recalling the key concept of pseudorandomness.

\begin{defn} Let $\bc \in \bZ^s$, and let $\nu: \bZ \to [0,\infty)$ be supported on $[N]$. We say that $\nu$ is \emph{$\bc$-pseudorandom} if, for any $\del > 0$, there exists $c(\del) > 0$ such that if $0 \le f \le \nu$ then
\[
\| f \|_1 \ge \del \| \nu \|_1 \quad \Longrightarrow \quad
\sum_{\bc \cdot \bx = 0} \prod_{i \le s} f(x_i) \ge c(\del) \sum_{\bc \cdot \bx = 0} \prod_{i \le s} \nu(x_i).
\]
\end{defn}

For the pseudorandomness of $1_{[N]}$, Browning and Prendiville infer a quantitative rate $c(\del)$ from earlier work of Bloom \cite{Blo2012}. We now update this rate by incorporating recent advances due to Ko\'sciuszko \cite{Kos} and Filmus, Hatami, Hosseini, and Kelman \cite{FHHK2024}. The latter generalises the aforementioned breakthrough of Kelley and Meka \cite{KM2023}.

\begin{lemma}
\label{lem4.2}
Let $\tau(3)=9$, and set $\tau(s)=7$ for all $s\geqslant 4$. Let $s \ge 3$ and let $c_1, \ldots, c_s$ be non-zero integers summing to zero. Then $1_{[N]}$ is $\bc$-pseudorandom with quantitative rate
\[
c(\del) = \exp(-C \log^{\tau(s)}(2/\del)),
\]
where $C$ is a large, positive constant depending on $\bc$ alone.
\end{lemma}

\begin{proof}
Throughout this proof, the letter $C$ will denote a sufficiently large positive constant which depends on $\bc$ alone. The value of $C$ may change from line to line.
    
We claim that it suffices to show that for all $N\in\N$ and all 
$
\del \ge 2/N,
$
if $A\subseteq[N]$ satisfies $|A|\geqslant\delta N$, then
\begin{equation}
\label{eqn4.1}
|\{\bx\in A^s: \bc\cdot\bx = 0\}| \geqslant N^{s-1}\exp(-C\log^{\tau(s)}(2/\delta)).
\end{equation}
To see this, let $f:[N]\to[0,1]$ with $\lVert f\rVert_1 \geqslant \delta N$, where $N \in \bN$ and $\del > 0$, and let $A = \{ n\in[N]: f(n) \geqslant \del/2\}$. By averaging, we have $|A|\geqslant \del N/2$.
Since $f\geqslant (\delta/2)1_A$, if $N\geqslant 2/\del$ then we infer the required bound
\begin{equation}
\label{eqn4.2}
\sum_{\bc \cdot \bx = 0} \prod_{i \le s} f(x_i)
\geqslant \exp(-C\log^{\tau(s)}(2/\delta)) \sum_{\bc \cdot \bx = 0} \prod_{i \le s} 1_{[N]}(x_i).
\end{equation}
If instead $2/\del > N$, then \eqref{eqn4.2} follows from one of the power mean inequalities:
\[
\sum_{\bc \cdot \bx = 0} \prod_{i \le s} f(x_i) \ge \sum_{n \le N} f(n)^s \ge N (N^{-1} \| f \|_1)^s \ge \del^s N.
\]

It therefore remains to verify \eqref{eqn4.1}. For $s\geqslant 4$, this follows immediately from \cite[Theorem 3]{Kos}. 

\bigskip

For $s=3$, we need to perform some additional manoeuvres to extract \eqref{eqn4.1} from \cite{FHHK2024}. We begin with \cite[Theorem 1.5 (ii)]{FHHK2024}, which tells us that, for any prime $p>\lVert \bc\rVert_1$, if $B\subseteq\Z/p\Z$ does not admit a non-trivial solution to $\bc\cdot\bx \equiv 0 \mmod{p}$ with $\bx \in B^s$, then
\begin{equation*}
|B|/p\leqslant \exp(-C^{-1}\log^{1/9}p).
\end{equation*}
To ensure that the conditions of the cited theorem hold, one uses the parametrisation $({(c_1+c_2)x + c_2z, (c_1+c_2)y - c_1z, c_1x + c_2y})$ for solutions to $\bc\cdot\bx = 0$ over $\Z/p\Z$. This parametrisation ``captures'' solutions to $\bc\cdot\bx = 0$, in the sense that every solution is counted exactly $p$ many times.
The ``underlying graph'' is then a complete graph of order $3$, which is ``$2$-degenerate'' --- see \cite[Example 1.7]{FHHK2024} and the paragraph preceding it.
    
If we choose $\lVert \bc\rVert_1 N < p \le 2 \lVert \bc\rVert_1 N$, then $\bx\in[N]^s$ satisfies $\bc\cdot\bx=0$ if and only if $\bc\cdot\bx \equiv 0 \mmod{p}$. Hence, by embedding $[N]$ into $\Z/p\Z$ in the usual way, the cited theorem implies --- upon rearranging --- that there exists a positive integer $k \leqslant  \exp(C\log^{9}(2/\delta))$ such that every $S\subseteq[k]$ with $|S|\geqslant \delta k/2$ contains a non-trivial solution to $\bc\cdot\bx = 0$. Since the entries of $\bc$ sum to zero, an averaging argument of Varnavides~\cite{Var1959} --- with the calculation as in \cite[Proof of Lemma 6.8]{Gre2005} --- reveals that if $A \subseteq \bZ/p\bZ$ and $|A| \geqslant \eta p \ge 1$ then
\begin{equation*}
|\{\bx\in A^s: \bc\cdot\bx = 0\}| \geqslant \frac{\eta p^2}{2k^2} \geqslant N^2\exp(-C\log^{9}(2/\eta)).
\end{equation*}
This verifies \eqref{eqn4.1}, upon choosing $\eta = \delta/(2\|\bc\|_1)$.
\end{proof}

This lemma leads to the following refinement of \cite[Proposition 2.8]{BP2017}.

\begin{thm}
\label{thm4.3}
Let $s\geqslant 3$ and let $\tau(s)$ be as in the statement of Lemma~\ref{lem4.2}.
Let $c_1, \ldots, c_s$ be non-zero integers summing to zero. Let $\cK$ be a finite union of $k$ many proper subspaces of the hyperplane defined by $\bc \cdot \bx = 0$. Let $\nu: \bZ \to [0,\infty)$ be supported on $[N]$, with the following properties:
\begin{enumerate}[(i)]
\item (Fourier decay) For some $\tet \in (0,1]$,
\[
\| \hat \nu - \hat 1_{[N]} \|_\infty \le \tet N.
\]
\item (Fourier restriction) For some $t \in [s-1,s)$,
\[
\sup_{|\psi| \le \nu} |\hat \psi(\alp)|^t \d \alp \ll_t \| \nu \|_1^t N^{-1}.
\]
\item ($\cK$-trivial saving) For some $\eta > 0$,
\[
\sum_{\bx \in \cK} \prod_{i \le s} \nu(x_i) \ll_{s,k} \| \nu \|_1^s N^{-1-\eta}.
\]
\end{enumerate}
Let $\cA \subseteq \supp(\nu)$ be such that if $\bx \in \cA^s$ and $\bc \cdot \bx = 0$, then $\bx \in \cK$. Then, for some sufficiently large constant $C = C(\bc, t, k, \eta) > 1$,
\[
\sum_{n \in \cA} \nu(n) \le \frac{N}{\exp\left(C^{-1} \min\{\log\log (2/\theta), \log N\}^{1/\tau(s)}\right)}.
\]
\end{thm}

\begin{proof}
After inserting the bounds given in Lemma \ref{lem4.2} into \cite[Equations (2.4) and (2.5)]{BP2017}, the theorem follows with the same proof as \cite[Proposition 2.8]{BP2017}.
\end{proof}

\section{Applying transference}
\label{sec5}

In this final section, we use our Fourier decay and restriction estimates to prove Theorems \ref{thm1.4} and \ref{thm1.8}.

\subsection{Bounding the weighted number of trivial solutions}

Let $c_1, \ldots, c_s$ be non-zero integers, where $s \ge 3$.  Let $1 < c < 2$, and put $\gam = 1/c$.

\begin{lemma} 
\label{lem5.1}
Let
\[
\cK = \{ \bn \in \bZ^s: \bc \cdot \bn = 0, \quad
n_i = n_j \text{ for some } i \ne j \}.
\]
Then, for some $\eta > 0$,
\[
\sum_{\bn \in \cK} \prod_{i\le s} \nu(n_i) \ll_{c,\eta} N^{s - 1 - \eta}.
\]
\end{lemma}

\begin{proof} Put
\[
e_1 = c_1, \ldots, e_{s-2} = c_{s-2}, \qquad
e_{s-1} = c_{s-1} + c_s.
\]
By symmetry, it suffices to prove that
\begin{equation}
\label{eqn5.1}
\sum_{\bn \in K} \prod_{i \le s} \nu(n_i) \ll N^{s-1-\eta},
\end{equation}
where
\begin{align*}
K &= \{ \bn \in \bZ^s: \bc \cdot \bn = 0, \quad
n_{s-1} = n_s \} \\
&= \{ (n_1, n_2, \ldots, n_{s-1}, n_{s-1}) \in \bZ^s:
e_1 n_1 + \cdots + e_{s-1} n_{s-1} = 0 \}.
\end{align*}
Recall that $\| \hat \nu \|_\infty \ll N$ and $\| \nu \|_\infty \ll N^{1-\gam}$. By orthogonality,  H\"older's inequality, and periodicity,
\begin{align*}
\sum_{\bn \in K} \prod_{i \le s} \nu(n_i) &\ll N^{1-\gam}
\sum_{\bn \in K} \prod_{i \le s - 1} \nu(n_i) \\
&= N^{1-\gam}
\int_\bT \hat \nu(e_1 \alp) \cdots \hat \nu(e_{s-2}\alp) \hat \nu (e_{s-1} \alp) \d \alp \\
&\ll N^{s-2-\gam} \int_\bT |\hat 
\nu(\alp)|^2 \d \alp \ll N^{s - 2 \gam} = N^{s - 1 - \eta},
\end{align*}
where $\eta = 2\gam - 1 > 0$. This confirms \eqref{eqn5.1}, thereby completing the proof of Lemma \ref{lem5.1}.
\end{proof}

\subsection{Proof of Theorem \ref{thm1.4}}

We establish the restriction estimate \eqref{eqn3.3} by dividing the analysis into two cases. If $s \le 4$, then we have $1 < c < c^*(s)$. Our definition (\ref{eqn1.1}) of $c^*(s)$ then guarantees that there exist $\chi, t_0, t \in \bR$ satisfying \eqref{eqn3.2} and
\[
s - 1 < t_0 < t < s, \qquad
0 < \chi < \frac{2}{5c} - \frac15.
\]
Hence, Proposition \ref{prop2.1} provides us with the Fourier decay estimate \eqref{eqn3.1}. Then, by Proposition \ref{prop3.1}, we have the restriction estimate \eqref{eqn3.3}. 
If instead $s \ge 5$ then, by Propositions \ref{prop2.1} and \ref{prop3.2}, we again have \eqref{eqn3.3} for some $t \in (s-1,s)$. 
For any $s \ge 3$, it follows from \eqref{eqn3.1} that
\[
\left\| \frac{\hat \nu}{\| \nu \|_1} - \frac{\hat{1}_{[N]}}{\| 1_{[N]} \|_1} \right \|_\infty \le \frac1M,
\]
for some $M \asymp N^\chi$.

For $F_1, \ldots, F_s: \bZ \to \bC$ compactly supported, define
\[
T(F_1, \ldots, F_s) = \sum_{\bc \cdot \bn = 0} F_1(n_1) \cdots F_s(n_s). 
\]
We abbreviate
\[
T(F) = T(F,\ldots,F).
\]
For $i \in [r]$, let
\[
f_i(n) = \begin{cases}
\phi'(n)^{-1}, &\text{if } n \in C_i \\
0, &\text{otherwise},
\end{cases}
\]
so that $\sum_{i \le r} f_i = \nu$. By the modelling lemma \cite[Proposition 14.1]{CLP2021}, there exist $g_1, \ldots, g_r: [N] \to [0, \infty)$ such that 
\[
g_1 + \cdots + g_r = (1 + M^{-1/2}) 1_{[N]}
\]
and
\[
\left \| \frac{\hat f_i}{\| \nu \|_1}
- \frac{\hat g_i}{N} \right \|_\infty \ll (\log N)^{-1/(t+2)}.
\]
The generalised von Neumann lemma \cite[Lemma C.3]{CLP2021} now gives
\[
T(f_i/\|\nu\|_1) - T(g_i/N)
\ll N^{-1} (\log N)^{-(s-t)/(t+2)} \qquad (1 \le i \le r).
\]

By functional FGR \cite[Lemma 15.2]{CLP2021}, there exists $k$ such that
\[
T(g_k) \gg N^{s-1}.
\]
Since $\| \nu \|_1 \asymp N$, for this value of $k$ we therefore have
\[
T(f_k) \gg N^{s-1}.
\]
By Lemma \ref{lem5.1}, this shows that there exists $\bx \in C_k^s$ with \eqref{eqn1.2}.

\subsection{Proof of Theorem \ref{thm1.8}}

Let $A\subseteq\PS_c(N)$ be as in the statement of Theorem \ref{thm1.8}. 
We compute that
\begin{align*}
\sum_{n \in A} \nu(n) & \ge \sum_{x \le |A|} cx^{c-1} \geqslant \int_{0}^{|A|}cx^{c-1}\d x = |A|^c.
\end{align*}
With the same constants $\chi > 0$ and $t \in (s-1, s)$ from the previous subsection, we again have the Fourier decay estimate \eqref{eqn3.1} and the restriction estimate \eqref{eqn3.3}. Thus, by Lemma \ref{lem5.1},
all of the sufficient conditions in Theorem \ref{thm4.3} are met, delivering the claim. 

\providecommand{\bysame}{\leavevmode\hbox to3em{\hrulefill}\thinspace}


\begin{thebibliography}{50}

\bibitem{AEM2021}
C. Aistleitner, D. El-Baz, and M. Munsch, \emph{A pair correlation problem, and counting lattice points with the zeta function}, Geom. Funct. Anal. \textbf{31} (2021), 483--512.

\bibitem{Beh1946}
F. A. Behrend, \emph{On sets of integers which contain no three terms in arithmetical progression},
Proc. Nat. Acad. Sci. U. S. A. \textbf{32} (1946), 331--332.

\bibitem{Blo2012}
T. F. Bloom, \emph{Translation invariant equations and the method of Sanders},
Bull. Lond. Math. Soc. \textbf{44} (2012), 1050--1067.

\bibitem{BS}
T. F. Bloom and O. Sisask, \emph{An improvement to the Kelley-Meka bounds on three-term arithmetic progressions}, arXiv:2309.02353.

\bibitem{B1989}
J. Bourgain, \emph{On $\Lambda(p)$-subsets of squares}, Israel J. Math. \textbf{67} (1989), 291--311.

\bibitem{BP2017}
T. D. Browning and S. M. Prendiville, \emph{A transference approach to a Roth-type theorem in the squares}, Int. Math. Res. Not. \textbf{2017,} 2219--2248.

\bibitem{Cha2022}
J. Chapman, \emph{Partition regularity for systems of diagonal equations}, Int. Math. Res. Not. \textbf{2022,} 
13272--13316.

\bibitem{CC1}
J. Chapman and S. Chow, \emph{Arithmetic Ramsey theory over the primes}, Proc. Roy. Soc. Edinburgh Sect. A, to appear.

\bibitem{CC2}
J. Chapman and S. Chow, \emph{Generalised Rado and Roth criteria}, Ann. Sc. Norm. Super. Pisa Cl. Sci. (5), to appear.

\bibitem{Cho2018}
S. Chow, \emph{Roth--Waring--Goldbach}, Int. Math. Res. Not. \textbf{2018,} 2341--2374.

\bibitem{CLP2021}
S. Chow, S. Lindqvist, and S. Prendiville, \emph{Rado's criterion over squares and higher powers}, J. Eur. Math. Soc. \textbf{23} (2021), 1925--1997.

\bibitem{FHHK2024}
Y. Filmus, H. Hatami, K. Hosseini, and E. Kelman, \emph{Sparse Graph Counting and Kelley--Meka Bounds for Binary Systems}. In: 2024 IEEE 65th Annual Symposium on Foundations of Computer Science (FOCS 2024), 1559--1578, IEEE Computer Society, Los Alamitos, CA, 2024.

\bibitem{FGR1988}
P. Frankl, R. L. Graham, and V. R\"odl, \emph{Quantitative Theorems for Regular
Systems of Equations},
J. Combin. Theory Ser. A \textbf{47} (1988), 246--261.

\bibitem{FKM}
N. Frantzikinakis, O. Klurman, and J. Moreira, \emph{Partition regularity of Pythagorean pairs}, 	arXiv:2309.10636.

\bibitem{Gra2007}
R. Graham, \emph{Some of my favorite problems in Ramsey theory}. In: Combinatorial Number Theory, 229--236, Walter de Gruyter \& Co., Berlin, 2007.

\bibitem{Gra2008}
R. Graham, \emph{Old and new problems in Ramsey theory}. In: Horizons of Combinatorics, 105--118, Bolyai Soc. Math. Stud. \textbf{17,} Springer, Berlin, 2008.

\bibitem{Gre2005}
B. J. Green, \emph{Roth's theorem in the primes}, Ann. of Math. (2) \textbf{161} (2005), no. 3, 1609--1636.

\bibitem{GL2019}
B. J. Green and S. Lindqvist, \emph{Monochromatic solutions to $x + y = z^2$}, Canad. J. Math. \textbf{71} (2019), 579--605.

\bibitem{HB1983}
D. R. Heath-Brown, \emph{The Pjatecki\u{\i}--\u{S}apiro prime number theorem}, J. Number Theory \textbf{16} (1983), 242--266.

\bibitem{HKM2016}
M. J. H. Heule, O. Kullmann, and V. W. Marek, \emph{Solving and verifying the Boolean Pythagorean
triples problem via cube-and-conquer}. In: Theory and Applications of Satisfiability
Testing---SAT 2016, 228--245, Lecture Notes in Comput. Sci. \textbf{9710,}
Springer, Cham, 2016.

\bibitem{KM2023}
Z. Kelley and R. Meka, \emph{Strong Bounds for 3-Progressions}, In: 2023 IEEE 64th Annual Symposium on Foundations of Computer Science (FOCS 2023), 933--973, IEEE Computer Society, Los Alamitos, CA, 2023. 

\bibitem{Kos}
T. Ko\'sciuszko, \emph{Counting solutions to invariant equations in dense sets}, 	arXiv:2306.08567.

\bibitem{Lam2016}
E. Lamb, \emph{Maths proof smashes size record}, Nature \textbf{534} (2016), 17--18.

\bibitem{Le2012}
T. H. L\^e, \emph{Problems and results on intersective sets}. In: Combinatorial and additive number theory---CANT 2011 and 2012, 115--128, Springer Proc. Math. Stat. \textbf{101,}
Springer, New York, 2014.

\bibitem{Mir2015}
M. Mirek, \emph{Roth's theorem in Piatetski-Shapiro primes}, Rev. Mat. Iberoam. \textbf{31} (2015), 617--656.

\bibitem{Mor2017}
J. Moreira, \emph{Monochromatic sums and products in $\bN$}, Ann. of Math. (2) \textbf{185} (2017), 1069--1090.

\bibitem{Pac2018}
P. P. Pach, \emph{Monochromatic solutions to $x + y = z^2$ in the interval $[N, cN^4]$}, Bull. Lond. Math. Soc. \textbf{50} (2018), 1113--1116.

\bibitem{PS1953}
I. Piatetski-Shapiro, \emph{On the distribution of prime numbers in sequences of the form $\lfloor f(n) \rfloor$}, Math. Sbornik
\textbf{33} (1953), 559--566.

\bibitem{Pre2021}
S. Prendiville, \emph{Counting monochromatic solutions to diagonal Diophantine equations}, Discrete Anal. (2021), Paper No. 14, 47 pages.

\bibitem{Rad1933}
R. Rado, \emph{Studien zur Kombinatorik}, Math. Z. \textbf{36} (1933), 424--470.

\bibitem{RS2006}
O. Robert and P. Sargos, \emph{Three-dimensional exponential sums with monomials}, J. Reine Angew. Math. \textbf{591} (2006), 1--20.

\bibitem{Rot1953}
K. F. Roth, \emph{On certain sets of integers}, J. London Math. Soc. \textbf{28} (1953),
104--109.

\bibitem{SY2019}
K. Saito and Y. Yoshida, \emph{Arithmetic Progressions in the Graphs of Slightly Curved Sequences}, J. Integer Seq. \textbf{22} (2019), Article 19.2.1.

\bibitem{Sch1916}
I. Schur, \emph{\"Uber die Kongruenz $x^m + y^m \equiv z^m$ \rm{(mod $p$)}}, Jahresber. Deutsch. Math.-Verein. \textbf{25} (1916), 114--117.

\bibitem{Sze1975}
E. Szemer\'edi, \emph{On sets of integers containing no $k$ elements in arithmetic progression}, Acta Arith. \textbf{27} (1975), 199--245.

\bibitem{Var1959} 
P. Varnavides, \emph{On certain sets of positive density}, J. London Math. Soc. \textbf{34}
(1959), 358--360.

\bibitem{ZZ2025}
Q. Zhang and R. Zhang, \emph{Roth-type Theorem for High-power System in Piatetski-Shapiro primes}, Front. Math. \textbf{20} (2025), 67--86.

\end{thebibliography}
\end{document}